\documentclass[12pt]{article}
\usepackage{latexsym}
\usepackage{amsmath}
\usepackage{amssymb}
\usepackage{amscd}
\usepackage{array}
\usepackage{comment}
\usepackage{amsthm}
\usepackage{amsfonts}
\usepackage{enumerate}
\usepackage{hyperref}
\usepackage{stackrel}
\usepackage{pdflscape}
\newtheorem{theorem}{Theorem}[]
\newtheorem{proposition}[theorem]{Proposition}
\newtheorem{lemma}[theorem]{Lemma}

\theoremstyle{definition}
\newtheorem{definition}[theorem]{Definition}

\newtheorem{examples}[theorem]{Examples}

\newtheorem{remark}[theorem]{Remark}

\newcommand{\Gal}{\mathrm{Gal}}

\newcommand{\Stab}{\mathrm{Stab}}
\newcommand{\Hol}{\mathrm{Hol}}

\newcommand{\Sym}{\operatorname{Sym}}

\newcommand{\Nor}{\mathrm{Nor}}

\newcommand{\ord}{\mathrm{ord}}

\newcommand{\End}{\operatorname{End}}

\newcommand{\Aut}{\operatorname{Aut}}

\newcommand{\Id}{\operatorname{Id}}
\newcommand{\Syl}{\operatorname{Syl}}
\newcommand{\wL} {{\widetilde{L}}}

\begin{document}

\Large
\begin{center}
{\bf Automatic realization of Hopf Galois structures}

\large
\vspace{0.4cm}
Teresa Crespo

\vspace{0.2cm}

\normalsize{Departament de Matem\`atiques i Inform\`atica, Universitat de Barcelona, \\ Gran Via de les Corts Catalanes 585, 08007 Barcelona, Spain, \\
e-mail: teresa.crespo@ub.edu}
\end{center}

\let\thefootnote\relax\footnotetext{{\bf 2010 MSC:} 12F10, 16T05, 20B05, 20B35, 20D20, 20D45 \\  Keywords: Separable field extensions; Hopf algebras; Hopf Galois structures; regular groups; left braces.}

\normalsize

\begin{abstract} We consider Hopf Galois structures on a separable field extension $L/K$ of degree $p^n$, for $p$ an odd prime number, $n\geq 3$. For $p > n$,  we prove that $L/K$ has at most one abelian type of Hopf Galois structures. For a nonabelian group $N$ of order $p^n$, with commutator subgroup of order $p$, we prove that if $L/K$ has a Hopf Galois structure of type $N$, then it has a Hopf Galois structure of type $A$, where $A$ is an abelian group of order $p^n$ and having the same number of elements of order $p^m$ as $N$, for $1\leq m \leq n$.
\end{abstract}

\section{Introduction}
A Hopf Galois structure on a finite extension of fields $L/K$ is a pair $(H,\mu)$, where $H$ is
a finite cocommutative $K$-Hopf algebra  and $\mu$ is a
Hopf action of $H$ on $L$, i.e a $K$-linear map $\mu: H \to
\End_K(L)$ giving $L$ a left $H$-module algebra structure and inducing a $K$-vector space isomorphism $L\otimes_K H\to\End_K(L)$.
Hopf Galois structures were introduced by Chase and Sweedler in \cite{C-S}.
For separable field extensions, Greither and
Pareigis \cite{G-P} give the following group-theoretic
equivalent condition to the existence of a Hopf Galois structure.

\begin{theorem}\label{G-P}
Let $L/K$ be a separable field extension of degree $g$, $\wL$ its Galois closure, $G=\Gal(\wL/K), G'=\Gal(\wL/L)$. Then there is a bijective correspondence
between the set of isomorphism classes of Hopf Galois structures on $L/K$ and the set of
regular subgroups $N$ of the symmetric group $S_g$ normalized by $\lambda_G(G)$, where
$\lambda_G:G \hookrightarrow S_g$ is the monomorphism given by the action of
$G$ on the left cosets $G/G'$.
\end{theorem}

For a given Hopf Galois structure on a separable field extension $L/K$ of degree $g$, we will refer to the isomorphism class of the corresponding group $N$ as the type of the Hopf Galois
structure. Given a regular subgroup $N$ of $S_g$, normalized by $\lambda_G(G)$, the corresponding Hopf Galois structure $(H,\mu)$ is obtained by Galois descent.

Childs \cite{Ch1} gives an equivalent  condition to the existence of a Hopf Galois structure introducing the holomorph of the regular subgroup $N$ of $S_g$. Let $\lambda_N:N\to \Sym(N)$ be the morphism given by the action of
$N$ on itself by left translation. The holomorph $\Hol(N)$ of $N$ is the normalizer of $\lambda_N(N)$ in $\Sym(N)$. As abstract groups, we have $\Hol(N)=N\rtimes \Aut(N)$. We state the more precise formulation of Childs' result due to Byott \cite{B} (see also \cite{Ch2} Theorem 7.3).

\begin{theorem}\label{theoB} Let $G$ be a finite group, $G'\subset G$ a subgroup and $\lambda_G:G\to \Sym(G/G')$ the morphism given by the action of
$G$ on the left cosets $G/G'$.
Let $N$ be a group of
order $[G:G']$ with identity element $e_N$. Then there is a
bijection between
$$
{\cal N}=\{\alpha:N\hookrightarrow \Sym(G/G') \mbox{ such that
}\alpha (N)\mbox{ is regular}\}
$$
and
$$
{\cal G}=\{\beta:G\hookrightarrow \Sym(N) \mbox{ such that }\beta
(G')\mbox{ is the stabilizer of } e_N\}
$$
Under this bijection, if $\alpha, \alpha' \in {\cal N}$ correspond to
$\beta, \beta' \in {\cal G}$, respectively, then $\alpha(N)=\alpha'(N)$ if and only if $\beta(G)$ and $\beta'(G)$ are conjugate by an element of $\Aut(N)$; and  $\alpha(N)$ is normalized by
$\lambda_G(G)$ if and only if $\beta(G)$ is contained in the
holomorph $\Hol(N)$ of $N$.
\end{theorem}

Recently a relationship has been found between Hopf Galois structures and an algebraic structure called brace. Classical braces were introduced by W. Rump \cite{R}, as a generalisation of radical rings, in order to study the non-degenerate involutive set-theoretic solutions of the quantum Yang-Baxter equation. Recently, skew braces were introduced by Guarnieri and Vendramin \cite{G-V} in order to study the non-degenerate (not necessarily involutive) set-theoretic solutions. This connection is further exploited in \cite{SV}, where the relation of braces with other algebraic structures is established.

\begin{definition} A left brace is a set $B$ endowed with two binary operations $\cdot$ and $\circ$ such that $(B,\cdot)$ and $(B,\circ)$ are groups and the two operations are related by the brace property

$$ a\circ(b\cdot c)=(a\circ b)\cdot a^{-1} \cdot (a\circ c), \text{\ for all \ } a,b,c \in B,$$

\noindent
where $a^{-1}$ denotes the inverse of $a$ in $(B,\cdot)$. The groups $(B,\cdot)$ and $(B,\circ)$ are called respectively the additive group and the multiplicative group of the brace $B$. The brace is called classical when its additive group is abelian, skew otherwise.

A map between braces is a brace morphism if it is a group morphism both between the additive and the multiplicative groups.

\end{definition}

The relation between braces and Hopf-Galois structures was first proved by Bachiller for classical braces (see \cite{Ba} Proposition 2.3) and generalized by Guarnieri and Vendramin to skew braces.

\begin{proposition}[\cite{G-V} Proposition 4.3] \label{GV}
Let $(N,\cdot)$ be a group. There is a bijective correspondence between isomorphism classes of left braces with additive group isomorphic to $(N,\cdot)$ and classes of regular subgroups of $\Hol(N)$ under conjugation by elements of $\Aut(N)$.
\end{proposition}




For a finite separable field extension $L/K$ we denote by $\widetilde{L}$ a normal closure of $L/K$, by $G'$ the Galois group of $\widetilde{L}/L$. We shall call a pair of groups $(G,N)$ \emph{realizable} if a separable field extension $L/K$ of degree $|N|$ such that $\Gal(\widetilde{L}/K) \simeq G$ has a Hopf Galois structure of type $N$. By Theorems \ref{G-P} and \ref{theoB}, a pair of groups $(G,N)$, such that $G$ has a subgroup $G'$ with $[G:G']=|N|$, is realizable if and only if there exists a group monomorphism $\beta:G\hookrightarrow \Hol(N)$ such that $\beta(G')$ is the stabilizer of $e_N$. In particular a pair of groups $(G,N)$ with $|G|=|N|$ is realizable if a Galois field extension $L/K$ with Galois group isomorphic to $G$ has a Hopf Galois structure of type $N$. In this case, by Theorem \ref{theoB} and Proposition \ref{GV}, $(G,N)$ is realizable if and only if there exists a left brace $B$ with additive group isomorphic to $N$ and multiplicative group isomorphic to $G$.

In this paper we obtain that if a pair of groups $(G,N)$ is realizable where $N$ is an abelian group of order $p^n$, with $n \geq 3$, $p$ a prime number, $p > n$, then no pair $(G,N')$ is realizable, where $N'$ is an abelian group of order $p^n$ nonisomorphic to $N$. This generalizes a result in \cite{C-C-F}. We also prove that for a nonabelian group $N$ of order $p^n$, with a commutator subgroup of order $p$, if a pair of groups $(G,N)$ is realizable, then $(G,A)$ is realizable, where $A$ is an abelian group of order $p^n$ and having the same number of elements of order $p^m$ as $N$, for $1\leq m \leq n$.

We refer the reader to \cite{Ha}, \cite{Hu} or \cite{F-A} for topics on finite $p$-groups.

\section{Hopf Galois structures of abelian type}

Let $p$ denote an odd prime number. We proved in \cite{CS2} Proposition 4 that if a separable extension of degree $p^n$ has a Hopf Galois structure of cyclic type, then it has no structure of noncyclic type. In the case of separable extensions of degree $p^3$, we obtained in \cite{CS2} Theorem 9 that, for $p >3$, the two abelian noncyclic types of Hopf Galois structures do not occur on the same extension. In this section we prove that two different abelian types of Hopf Galois structures do not occur on a separable extension of degree $p^n$, for $n \geq 3$, $p > n$. This result generalizes \cite{C-C-F}, Theorem 1, where the authors prove that if $(N, +)$ is a finite abelian
$p$-group of $p$-rank $m$ where $m + 1 < p$, then every regular abelian subgroup
of the holomorph of $N$ is isomorphic to $N$. As a consequence we obtain that two classical braces of order $p^n$, $p\geq n$, with isomorphic multiplicative group must have isomorphic additive groups.

We shall need to consider the $p$-Sylow subgroup $\Syl_p(G)$ of a transitive subgroup $G$ of the holomorph of a group of order $p^n$.

\begin{lemma}\label{trans} Let $G$ be a subgroup of $\Hol(N)$, for $N$ a group of order $p^n$, where $p$ is an odd prime number. Then $G$ is transitive if and only if $\Syl_p(G)$ is transitive.
\end{lemma}

\begin{proof}
Clearly, if $G$ is a subgroup of $\Hol(N)$, then $\Syl_p(G)$ is a subgroup of $\Syl_p(\Hol(N))$. Now, $G$ is transitive if and only if $[G:G\cap \Stab_{\Hol(N)}(e_N)]=p^n$. We have the following equalities between indices.

$$\begin{array}{r}[G:\Syl_p(G)\cap \Stab_{\Hol(N)}(e_N)]=[G:\Syl_p(G)][\Syl_p(G):\Syl_p(G)\cap \Stab_{\Hol(N)}(e_N)]\\
=[G:G\cap \Stab_{\Hol(N)}(e_N)][G\cap \Stab_{\Hol(N)}(e_N):\Syl_p(G)\cap \Stab_{\Hol(N)}(e_N)].\end{array}$$

\noindent
Since $[G:\Syl_p(G)]$ and $[G\cap \Stab_{\Hol(N)}(e_N):\Syl_p(G)\cap \Stab_{\Hol(N)}(e_N)]$ are prime to $p$, we obtain that $G$ is transitive if and only if $\Syl_p(G)$ is transitive.
\end{proof}

\begin{theorem}\label{abelian}
Let $N_1,N_2$ be abelian groups of order $p^n$, with $n \geq 3$, $p > n$. Let $G$ be a group.
If the pairs $(G,N_1)$ and $(G,N_2)$ are realizable, then $N_1 \simeq N_2$.
\end{theorem}

In order to prove the theorem we shall use the following lemma.

\begin{lemma}\label{order} Let $N$ be an abelian group of order $p^n$, $p>n$, $G$ a transitive subgroup of $\Hol(N)$, of order $|G|=p^m, m\geq n$. We consider the surjective map
$\pi: G \rightarrow N, (a,\varphi) \mapsto a.$
If $\pi(a,\varphi)=a$, then $a^{p^k}= e_N \Leftrightarrow (a,\varphi)^{p^k} \in \Stab_{\Hol(N)}(e_N)$.
\end{lemma}

\begin{proof} Since the product in $\Hol(N)=N\rtimes \Aut(N)$ is defined by $(x,\varphi)(y,\psi)=(x\varphi(y),\varphi\psi)$, the map $\gamma:\Hol(N) \rightarrow \Aut(N)$ is a group morphism. Hence $\gamma(G)$ is a subgroup of $\Aut(N)$ of order a divisor of $|G|$. We use the notation in \cite{CCC}, where the funcions $\gamma$ are used to determine all Hopf-Galois structures on Galois field
extensions of degree $p^2q$, with $p$ and $q$ distinct primes, $p>2$. Since $N$ is normal in its holomorph, we have that $N\gamma(G)$ is a subgroup of $\Hol(N)$, of order a power of $p$, and thus a nilpotent group. In particular, since $|G|=p^n$, we have

\begin{equation}\label{com}
N \supset [N,\gamma(G)] \supset [N,\gamma(G),\gamma(G)] \supset \dots \supset [N,\underbrace{\gamma(G),\dots,\gamma(G)}_\text{$n$ times}]=\{1 \}.
\end{equation}

\noindent
We recall that the commutator of $a \in N$ and $\varphi \in \Aut(N)$ is $[a,\varphi]=(a^{-1},\Id)(e_N,\varphi^{-1})(a,\Id)(e_N,\varphi)$ \linebreak $=a^{-1} \varphi^{-1}(a)$.
For $(a,\varphi) \in \Hol(N)$, we have $(a,\varphi)^2=(a\varphi(a),\varphi^2)$ and, by induction on $k$, we obtain

$$
(a,\varphi)^k=(a\varphi(a)\dots \varphi^{k-1}(a),\varphi^k)=((\Id+\varphi+\dots+\varphi^{k-1})(a),\varphi^k),
$$

\noindent
where the sum is taken in the endomorphism ring of the abelian group $N$. In particular, for $(x,\gamma(x)) \in G$, we have

\begin{equation}\label{gam}
(x,\gamma(x))^k=((\Id+\gamma(x)+\dots+\gamma(x)^{k-1})(x),\gamma(x)^k).
\end{equation}

Consider the endomorphism $\delta(x)=-1+\gamma(x) \in \End(N)$. Note that for $x \in G, y \in N$, we have

$$\delta(x)(y)=y^{-1} \gamma(x)(y)=[y,\gamma^{-1}(x)] \in [N,\gamma(x)^{-1}] \subset [N,\gamma(G)],$$

\noindent
so that, by (\ref{com}), we have $\delta(x)^n=0$ in $\End(N)$. We can now rewrite (\ref{gam}) substituting $\gamma(x)=1+\delta(x)$, as

\begin{equation}\label{delta}
(x,\gamma(x))^k=((k\Id+\binom{k}{2}+\cdots + \binom{k}{k-1} \delta(x)^{k-2}+\delta(x)^{k-1})(x),\gamma(x)^k)
\end{equation}

\noindent
Since $\delta(x)^n=0$, from (\ref{delta}) we obtain

\begin{equation}\label{p}
(x,\gamma(x))^p=((p\Id+\binom{p}{2}\delta(x)+\cdots + \binom{p}{n} \delta(x)^{n-1})(x),\gamma(x)^p).
\end{equation}

Assume that $x$ has order $p$ in $N$. Since $p>n$, all binomial coefficients in (\ref{p}) are divisible by $p$. Now, since $N$ is abelian, $\Omega_1(N):=\{x \in N: x^p=1 \}$ is a subgroup of $N$ invariant under endomorphisms, hence $(p\Id+\binom{p}{2}\delta(x)+\cdots + \binom{p}{n} \delta(x)^{n-1})(x)=e_N$, which gives $(x,\gamma(x))^p \in \Stab_{\Hol(N)}(e_N)$.

Let $x$ now have order $p^k$ in $N$, for some $k>1$. We have

\begin{equation}\label{pk-1}
(x,\gamma(x))^{p^{k-1}}=((p^{k-1}\Id+\binom{p^{k-1}}{2}\delta(x)+\cdots + \binom{p^{k-1}}{n} \delta(x)^{n-1})(x),\gamma(x)^{p^{k-1}}).
\end{equation}

\noindent
Since $p>n$, all binomial coefficients are divisible by $p^{k-1}$. Since $x^{p^{k-1}} \in \Omega_1(N)$, we have $(p^{k-1}\Id+\binom{p^{k-1}}{2}\delta(x)+\cdots + \binom{p^{k-1}}{n} \delta(x)^{n-1})(x) \in \Omega_1(N)$, so that $(x,\gamma(x))^{p^{k}} \in \Stab_{\Hol(N)}(e_N)$. We claim that $(p^{k-1}\Id+\binom{p^{k-1}}{2}\delta(x)+\cdots + \binom{p^{k-1}}{n} \delta(x)^{n-1})(x) \neq 1$. Then the equivalence in the statement of the lemma will follow.

This relies on the following remark.

\begin{remark} Let $a \in N, a \neq 1$ and $S=\langle a \rangle ^{N\gamma(G)}$ be the smallest normal subgroup of $N\gamma(G)$ which contains $a$. Since $N\gamma(G)$ is nilpotent and $S$ is a nontrivial normal subgroup of $N\gamma(G)$, we have $S \varsupsetneq [S,N\gamma(G)]$. In particular, since $[S,N\gamma(G)]$ is also normalized by $N\gamma(G)$, we have that $a \not \in [S,N\gamma(G)]$.
\end{remark}

We apply this Remark to $a=x^{p^{k-1}} \neq 1$. Noting that $(\binom{p^{k-1}}{2}\delta(x)+\cdots + \binom{p^{k-1}}{n} \delta(x)^{n-1})(x) \in [S,\gamma(G)] \subset [S,N\gamma(G)]$, if we had $((p^{k-1}\Id+\binom{p^{k-1}}{2}\delta(x)+\cdots + \binom{p^{k-1}}{n} \delta(x)^{n-1})(x),\gamma(x)^{p^{k-1}})=1$, we would have $a \in [S,N\gamma(G)]$, a contradiction.
\end{proof}

\begin{proof}[Proof \nopunct]{\it of Theorem \ref{abelian}.}
If the pairs $(G,N_1)$ and $(G,N_2)$ are realizable, then $G$ has a subgroup $G'$ with $[G:G']=p^n$ and there exist group morphisms $\beta_i:G \rightarrow \Hol(N_i)$ with $\beta_i(G')=\Stab_{\Hol(N_i)}(e_{N_i}), i=1,2$. By Lemma \ref{trans}, we may assume that the order of $G$ is a $p$-power. Let $\pi_i:\Hol(N_i)=N_i\rtimes \Aut(N_i) \rightarrow N_i$ be the projection on the first factor, for $i=1,2$. Then the composition $\pi_i \circ \beta_i$ is an epimorphism and, for $x \in G$,  Lemma \ref{order} gives $x^{p^k} \in G'$ if and only if $(\pi_i \circ \beta_i)(x)^{p^k}=e_{N_i}, i=1,2$. Since the isomorphism type of a finite abelian group is determined by the number of elements of each order, the theorem is proved.
\end{proof}

\begin{remark} We note that the condition $p>n$ in Theorem \ref{abelian} is necessary. For example, by computation we obtain that a Galois extension with Galois group $C_9\times C_3\times C_3$ has Hopf Galois structures of types $C_9^2$ and $C_3^4$ and that a Galois extension with Galois group $C_3^4$ has Hopf Galois structures of type $C_9\times C_3 \times C_3$.
\end{remark}

\section{Hopf Galois structures of nonabelian type}
We proved in \cite{CS2} that if a separable field extension of degree $p^3$ has a nonabelian Hopf Galois structure of type $N$, then it has an abelian structure whose type has the same exponent as $N$. Here we generalize this result for separable field extensions of degree $p^n$. More precisely, let $A$ and $N$ be two groups of order $p^n$ such that $A$ is abelian, the commutator subgroup of $N$ has order $p$ and $A$ and $N$ have the same exponent and the same number of elements of order $p^m, 1\leq m \leq n$. With these hypothesis, if, for some group $G$, the pair $(G,N)$ is realizable, then the pair $(G,A)$ is also realizable. To prove this fact, by Theorem \ref{theoB}, it suffices to prove that $\Hol(A)$ contains a regular subgroup isomorphic to $N$ such that its normalizer in $\Hol(A)$ is equal to its normalizer in $\Sym(A)$, that is, has order equal to $\Hol(N)$.

Let $N$ be a group of order $p^n$ and assume that its commutator subgroup $[N,N]$ has order $p$. Then $N/[N,N]$ is an abelian group of order $p^{n-1}$ and $[N,N]$ is included in the center of $N$. Let $N/[N,N] = \oplus_{i=1}^{s}  \langle b_i \rangle$, with $b_i$ of order $p^{r_i}, 1\leq i \leq s$. Let $\pi:N \rightarrow N/[N,N]$ be the projection morphism. We choose $\beta_i \in N$ such that $\pi(\beta_i)=b_i$. Then $\{ \beta_i \}_{1\leq i \leq s}$ is a set of generators of $N$. Let $c$ be a generator of $[N,N]$. For a pair of indices $i,j, 1\leq i,j \leq s$, we have $\beta_i\beta_j\beta_i^{-1}\beta_j^{-1} \in [N,N]$, hence $\beta_i\beta_j\beta_i^{-1}=c^k \beta_j$, for some integer $k$. We have then $(\beta_i \beta_j)^p=c^{kp(p-1)/2} \beta_i^p\beta_j^p=\beta_i^p\beta_j^p$, since $p$ is odd.

We define an abelian group $A$ of order $p^n$ in the following way. If $\beta_i$ has the same order than $b_i$ for all $i=1,\dots,s$, then $A=\oplus_{i=1}^{s}  \langle \alpha_i \rangle \oplus \langle d \rangle$, with $\alpha_i$ of the same order as $\beta_i$ and $d$ of order $p$. If the order of $\beta_{i_0}$ is equal to $p$ times the order of $b_{i_0}$ for some $i_0$, then $A=\oplus_{i=1}^{s}  \langle \alpha_i \rangle$, with $\alpha_i$ of the same order as $\beta_i$. In this case, we put $d:=\beta_{i_0}^{\ord(\beta_{i_0})/p}$. In both cases, we have $A/\langle d \rangle \simeq N/[N,N]$ and $A$ has the same number of elements of order $p^m$ than $N$, $1\leq m \leq n$, since $\alpha_i$ and $\beta_i$ have the same order and $(\beta_i \beta_j)^p=\beta_i^p\beta_j^p$.

\begin{theorem}\label{nonab} Let $N$ and $A$ be groups of order $p^n$ as above. If for some group $G$, the pair $(G,N)$ is realizable, then $(G,A)$ is realizable.
\end{theorem}

\begin{proof}
We define automorphisms $\varphi_i$ of $A$ by

$$\varphi_i(d)=d, \, \varphi_i(\alpha_j)=d^{k/2} \alpha_j \text{\ if \ } \beta_i\beta_j\beta_i^{-1}=c^k \beta_j, 0\leq k <p,$$

\noindent
where $k/2$ is defined modulo $p$. We note that $\varphi_i(\alpha_j^p)=\alpha_j^p$, for all $j$, hence $\varphi_i$ is well defined and we have $\varphi_i^p=\Id$. Let us prove that the subgroup of $\Hol(A)$ generated by $\{(\alpha_i,\varphi_i)\}_{1\leq i \leq s}$ and $d$ is a regular subgroup of $\Hol(A)$ isomorphic to $N$. Since $\varphi_i(\alpha_i)=\alpha_i$, and $\varphi_i^p=\Id$, the order of $(\alpha_i,\varphi_i)$ is equal to the order of $\alpha_i$ which is equal to the order of $\beta_i$, and the order of $d$ is equal to the order of $c$. Now, if
$\beta_i\beta_j\beta_i^{-1}=c^k \beta_j$, we have

$$\begin{array}{lll}(\alpha_i,\varphi_i)(\alpha_j,\varphi_j)(\alpha_i,\varphi_i)^{-1}&=&
(\alpha_i,\varphi_i)(\alpha_j,\varphi_j)(\alpha_i^{-1},\varphi_i^{-1})\\ &=&
(\alpha_i\varphi_i(\alpha_j),\varphi_i\varphi_j)(\alpha_i^{-1},\varphi_i^{-1})\\&=&
(\alpha_i d^{k/2} \alpha_j,\varphi_i\varphi_j)(\alpha_i^{-1},\varphi_i^{-1})\\
&=&(d^{k/2} \alpha_i\alpha_j\varphi_i(\varphi_j(\alpha_i^{-1})),\varphi_i\varphi_j\varphi_i^{-1})\\
&=& (d^{k/2} \alpha_i\alpha_j d^{k/2} \alpha_i^{-1}, \varphi_j)\\
&=& (d^k \alpha_j,\varphi_j) \\ &=& d^k(\alpha_j,\varphi_j).
\end{array}
$$

\noindent
and $d(\alpha_i,\varphi_i)=(\alpha_i,\varphi_i)d$, since $\varphi_i(d)=d$. Hence the subgroup $N'$ of $\Hol(A)$ generated by $\{(\alpha_i,\varphi_i)\}_{1\leq i \leq s}$ and $d$ is isomorphic to $N$. Now, since $(\alpha_i,\varphi_i)^k=(\alpha_i^k,\varphi_i^k)$ and $(\alpha_i,\varphi_i)(\alpha_j,\varphi_j)=d^{k/2} (\alpha_i\alpha_j,\varphi_i\varphi_j)$, it is a regular subgroup.

We want to prove now that the normalizer $\Nor_{\Hol(A)}(N')$ of $N'$ in $\Hol(A)$ has order equal to $|\Hol(N')|$. Let us see that $A$ is included in $\Nor_{\Hol(A)}(N')$. Indeed, for $x \in A$, we have

$$x(\alpha_i,\varphi_i)x^{-1}=(x\alpha_i\varphi_i(x^{-1}),\varphi_i)=(d^r \alpha_i,\varphi_i)=d^r(\alpha_i,\varphi_i),$$

\noindent
for some integer $r$. Hence $A$ normalizes $N'$.

We consider now the bijective map $f: A \rightarrow G, d^r \prod \alpha_i^{r_i} \mapsto c^r \prod \beta_i^{r_i}$. It induces an injective group morphism $\widetilde{f}: \Aut G \rightarrow \Aut A, \chi \mapsto \widetilde{\chi}:=f^{-1} \circ \chi \circ f.$ Since $\widetilde{\chi}$ preserves the order and is bijective, it is indeed an automorphism of $A$. We shall see that $\widetilde{f}(\Aut G)$ normalizes $N'$. For $\widetilde{\chi} \in \widetilde{f}(\Aut G)$, we have

$$\widetilde{\chi} (\alpha_i,\varphi_i) \widetilde{\chi}^{-1}=(\widetilde{\chi}(\alpha_i),\widetilde{\chi}\varphi_i \widetilde{\chi}^{-1}).$$

\noindent
We shall prove that, if $\widetilde{\chi}(\alpha_i)=\alpha_1^{r_1}\dots\alpha_s^{r_s}$, then $\widetilde{\chi} \varphi_i \widetilde{\chi}^{-1}=\varphi_1^{r_1}\dots \varphi_s^{r_s}$. We consider $\alpha_j$ and write $\widetilde{\chi}(\alpha_j)=\alpha_1^{t_1}\dots\alpha_s^{t_s}$. We have $\varphi_i(\alpha_j)=c^{k/2} \alpha_j$ if $\beta_i\beta_j\beta_i^{-1}=c^k\beta_j$. In this case, we have

\begin{equation}\label{eqq}
(\beta_1^{r_1}\dots\beta_s^{r_s})(\beta_1^{t_1}\dots\beta_s^{t_s})=c^k (\beta_1^{t_1}\dots\beta_s^{t_s})(\beta_1^{r_1}\dots\beta_s^{r_s})
\end{equation}

\noindent
 Now

$$\begin{array}{l} \alpha_j \stackrel{\varphi_i}{\mapsto} d^{k/2} \alpha_j \stackrel{\widetilde{\chi}}{\mapsto} d^{k/2}\alpha_1^{t_1}\dots\alpha_s^{t_s} \\
\alpha_j \stackrel{\widetilde{\chi}}{\mapsto} \alpha_1^{t_1}\dots\alpha_s^{t_s} \stackrel{\varphi_1^{r_1}\dots \varphi_s^{r_s}}{\mapsto} d^{k/2}\alpha_1^{t_1}\dots\alpha_s^{t_s}, \end{array}$$

\noindent
taking into account (\ref{eqq}). We obtain then

$$\widetilde{\chi} (\alpha_i,\varphi_i) \widetilde{\chi}^{-1}= c^{\ell} (\alpha_1,\varphi_1)^{r_1}\dots(\alpha_s,\varphi_s)^{r_s},$$

\noindent
for some integer $\ell$. We have then $|\Nor_{\Hol(A)}(N')|=|\Hol(N')|$, as wanted.

\end{proof}

\begin{examples}
Theorem \ref{nonab} may be applied for instance to the following pairs of groups.
\begin{enumerate}[1)]
\item  $N=C_{p^{n-1}} \rtimes C_p, A= C_{p^{n-1}} \times C_p$, for $n\geq 3$;
\item $N=\langle a,b : a^{p^n}=1, b^{p^n}=1, bab^{-1} = a^{1+p^{n-1}} \rangle, A= C_{p^n} \times C_{p^n}$, for $n \geq 2$;
\item $N=\langle a,b,c : a^{p^n}=1, b^{p}=1, c^p=1, bab^{-1} = a, cac^{-1}=a, cbc^{-1}=ba^{p^{n-1}} \rangle, A= C_{p^n} \times C_p \times C_p$, for $n \geq 2$;
\item $N=\langle a,b,c : a^{p^n}=1, b^{p}=1, c^p=1, bab^{-1} = a, cac^{-1}=a^{1+p^{n-1}}, cbc^{-1}=b \rangle, A= C_{p^n} \times C_p \times C_p$, for $n \geq 2$;
\item $N=\langle a,b,c : a^{p^n}=1, b^{p}=1, c^p=1, bab^{-1} = a, cac^{-1}=ab, cbc^{-1}=b \rangle, A= C_{p^n} \times C_p \times C_p$, for $n \geq 2$.
\end{enumerate}
\end{examples}

\begin{remark} We note that the condition that the commutator subgroup of $N$ has order $p$ in Theorem \ref{nonab} is necessary. For example, the group $N:=\langle a,b,c : a^{p^2}=1, b^{p}=1, c^p=1, bab^{-1} = a^{1+p}, cac^{-1}=ab, cbc^{-1}=b \rangle$ has the same number of elements of order $p^2$ as $A:=C_{p^2} \times C_p \times C_p$, namely $p^4-p^3$, but $[N,N]=\langle a^p,b \rangle$ has order $p^2$. For $p=5$, we have checked with Magma that
$\Hol(A)$ has regular subgroups isomorphic to $N$ but the order of the normalizer of $N$ in $\Hol(A)$ is not equal to the order of $\Hol(N)$.
\end{remark}

\section*{Acknowledgements} I am very grateful to the referee for
his/her comments and detailed suggestions which helped me to obtain more general results than those in the previous version of
the manuscript.

This work was supported by grant PID2019-107297GB-I00 (MICINN).

\end{document}